\documentclass{amsart}
\usepackage{hyperref}
\usepackage{txfonts}
\usepackage{amsfonts}

\newtheorem{theorem}{Theorem}[section]
\newtheorem{lemma}[theorem]{Lemma}
\newtheorem{corollary}[theorem]{Corollary}

\theoremstyle{definition}

\theoremstyle{remark}
\newtheorem{remark}[theorem]{Remark}

\numberwithin{equation}{section}

\begin{document}

\title{On the fundamental group of complete manifolds with almost Euclidean volume growth}

\author{}
\address{}
\curraddr{}
\email{}
\thanks{}

\author{Jianming Wan}
\address{School of Mathematics, Northwest University, Xi'an 710127, China}
\email{wanj\_m@aliyun.com}
\thanks{}

%    General info

\date{}

\dedicatory{}


\begin{abstract}
It is proved that the fundamental group of a complete Riemannian manifold with nonnegative Ricci curvature and certain volume growth conditions is trivial or finite.
\end{abstract}
\maketitle

%% The correct journal style for \specialsection is all uppercase; a known bug
%% in amsart.cls prevents this, so input must be uppercase until it is fixed.
%\specialsection*{This is a Special Section Head}

%%%%%%%%%%%%%%%%%%%%%%%%%%%%%%%%%%%%%%%%%%%%%%%%%%%%%%%%%%%%%%%%%%%%%%%%

%%%%%%%%%%%%%%%%%%%%%%%%%%%%%%%%%%%%%%%%%%%%%%%%%%%%%%%%%%%%%%%%%%%%%%%%
\section{Introduction}
Throughout the paper $M$ denotes a complete noncompact Riemannian $n$-manifold with nonnegative Ricci curvature. Let $V_{p}(r)$ be the volume of the metric ball $B_{p}(r)$ origin at $p$ with radius $r$ in $M$. By Bishop-Gromov volume comparison , $\frac{V_{p}(r)}{\omega_{n}r^{n}}$ is a decreasing function, where $\omega_{n}$ is the volume of unit ball in $\mathbb{R}^{n}$. So the limit is existent as $r$ goes to infinite. Denote the \emph{volume growth} of $M$ by $$\alpha_{M}\coloneqq\lim_{r\rightarrow \infty}\frac{V_{p}(r)}{\omega_{n}r^{n}}.$$ The $\alpha_{M}$ is independent of $p$ and so a global geometric invariant. Moreover, the volume comparison also implies that $0\leq\alpha_{M}\leq1$ and $\alpha_{M}=1$ if and only if $M$ is isometric to $\mathbb{R}^{n}$. We say that $M$ has \emph{Euclidean volume growth} (or large volume growth) if $\alpha_{M}>0$.

The volume comparison theorem implies that $$0<\frac{V_{p}(2r)}{V_{p}(r)}\leq2^{n}$$ for all $r>0$ and Euclidean volume growth condition implies that $$\lim_{r\rightarrow \infty}\frac{V_{p}(2r)}{V_{p}(r)}=2^{n}.$$

The main result of this note is

\begin{theorem}\label{t1.1}
 Given $n$, there is a constant $C(n)<2^{n}$ such that if an open $n$- manifold $M$ satisfies \\
 1)
 \begin{equation}\label{f1.1}
 \frac{V_{p}(2r)}{V_{p}(r)}> C(n)
 \end{equation}
for some $p\in M$ and \textbf{all $r>0$}, then $M$ is simple connected.\\
 2)
\begin{equation}\label{f1.2}
\liminf_{r\rightarrow\infty}\frac{V_{p}(2r)}{V_{p}(r)}>C(n),
\end{equation}
for some $p\in M$, then the fundamental group $\pi_{1}(M)$ is finite.
\end{theorem}

We should note that even though $M$ has Euclidean volume growth, one can not deduce that $M$ is simple connected. So formula (\ref{f1.1}) holding for all $r>0$ is important.

Set $\epsilon(n)=n-\log_{2}\frac{C(n)+2^{n}}{2}$. We see immediately that 2) of Theorem \ref{t1.1} implies the following

\begin{corollary}\label{c1.2}
Given $n$, there is a constant $\epsilon(n)$ such that if an open $n$- manifold $M$ satisfies
\begin{equation}\label{f1.3}
\lim_{r\rightarrow\infty}\frac{V_{p}(r)}{r^{n-\epsilon}}>0,
\end{equation}
for some $p\in M$, $0\leq\epsilon<\epsilon(n)$, then $\pi_{1}(M)$ is finite.
\end{corollary}

This shows a gap phenomenon for a well-known result of Peter Li \cite{[L]} and Anderson \cite{[A]} states that  $\pi_{1}(M)$ is finite provided $M$ has Euclidean volume growth.

On the other hand, Anderson has proved that (see Theorem 1.1 in \cite{[A]}) under condition (\ref{f1.3}), every  finitely generated subgroup of $\pi_{1}(M)$ is actually of polynomial growth of order $\leq \epsilon<1$. In \cite{[W]} Bingye Wu proved that under condition (\ref{f1.3}) $\pi_{1}(M)$ is finitely generated.
But every infinite group of finitely generated has  polynomial growth of order at least 1 (I thank the referee for pointing out this fact.  See Section 3). So Corollary \ref{c1.2} is also a consequence of Anderson and Wu's results.

\textbf{Acknowledgment:} I would like to thank the referee for his (her) invaluable suggestions. The referee's explanation clarify some understanding of mine on Anderson's paper \cite{[A]}.

\section{A related volume ratio}
In this section we prove an estimate on the volume ratio $V_{p}(2r)/V_{p}(r)$ related to certain generated elements of $\pi_{1}(M)$ (Lemma \ref{l2.4} below). The main ingredients are Sormani's uniform cut lemma \cite{[So]} and some ideas due to Shen \cite{[Sh]}.

\subsection{A uniform estimate}
Let $g\in\pi_{1}(M,p)$ and $\pi: \widetilde{M}\rightarrow M$ be the universal cover. Following \cite{[So]}, we say that $\gamma$ is a \emph{minimal representative geodesic loop} (based at p) of $g$ if $\gamma=\pi\circ \tilde{\gamma}$, where $\tilde{\gamma}$ is a minimal geodesic connecting $\tilde{p}$ and $g\tilde{p}$. So $L(\gamma)=d_{\widetilde{M}}(\tilde{p}, g\tilde{p})$.

Given a group $G$, we say that $\{g_{1},g_{2},g_{3},\cdot\cdot\cdot\}$ is an \emph{ordered set of independent generators} of $G$ if every $g_{i}$ can not be expressed as the previous generators and their inverses, $g_{1},g_{1}^{-1},\cdot\cdot\cdot,g_{i-1},g_{i-1}^{-1}$.

In \cite{[So]} Sormani proved the following two lemmas.

\begin{lemma}\label{l2.1}
(halfway lemma) There exists an ordered set of independent generators $\{g_{1},g_{2},\cdot\cdot\cdot\}$ of $\pi_{1}(M,p)$ with minimal representative geodesic loops $\gamma_{k}$ of length $d_{k}$ such that $$d_{M}(\gamma_{k}(0),\gamma_{k}(d_{k}/2))=d_{k}/2.$$
In particular, we have a sequence of such generators if $\pi_{1}(M,p)$ is infinitely generated.
\end{lemma}

\begin{lemma}\label{l2.2}
(uniform cut lemma) Let $M^{n}$ ($n\geq3$) be a complete manifold with $Ric\geq0$. Let $\gamma$ be a noncontractible geodesic loop based at $p_{0}$ of length $L(\gamma)=2D$ such that
\begin{enumerate}
  \item If $\sigma$ based at $p_{0}$ is a loop homotopic to $\gamma$, then $L(\sigma)\geq2D$;
  \item The $\gamma$ is minimal on $[0,D]$ and $[D,2D]$.
\end{enumerate}
Then there is a constant $S_{n}$ depending on $n$ such that if $x\in\partial B_{p_{0}}(rD)$ where $r\geq 1+2S_{n}$, then
$$d(x,\gamma(D))\geq(r-1)D+4S_{n}D.$$
\end{lemma}

The main idea of proof of uniform cut lemma is to lift geodesic loop to the universal covering space and research carefully the related excess function. It contains a nice application of Abresch-Gromoll's estimate on excess function \cite{[AG]}. The above two lemmas allow her to show that Milnor conjecture holds for the manifold with so called small linear diameter growth.

Let $\gamma$ be a minimal representative geodesic loops based at $p$ of $L(\gamma)=d$ satisfying Lemma \ref{l2.1}. The below estimate is important for our purpose.

\begin{lemma}\label{l2.3}
Let $\sigma$ be a geodesic issuing from $p$ such that $\sigma(t)$ is minimal on $[0, d)$. Then there is a constant $S(n)$ such that
$$h \coloneqq d_{M}(\gamma(d/2),\sigma|_{[0,d)})\geq S(n)d.$$
\end{lemma}

\begin{proof}
We set $h_{1} \coloneqq d_{M}(\gamma(d/2),\sigma|_{[0,d/2)})$ and $h_{2} \coloneqq d_{M}(\gamma(d/2,\sigma|_{[d/2,d)}).$
By Lemma \ref{l2.2}, we have $$H=d_{M}(\gamma(d/2),\sigma(d/2))\geq 2S(n)d,$$
where $S(n)$ is a universal constant, $S(n)=\frac{n}{n-1}\frac{1}{4}\frac{1}{3^{n}}(\frac{n-2}{n-1})^{n-1}.$

Suppose that $h_{1}=d_{M}(\gamma(d/2),\sigma(r_{0}))$. By the triangle inequality one has $$h_{1}\geq \frac{d}{2}-r_{0}$$
and $$h_{1}\geq H-(\frac{d}{2}-r_{0}).$$ Then $$h_{1}\geq H/2\geq S(n)d.$$
We also note that $d_{M}(\gamma(d/2),\sigma(d))\geq d/2.$ So similarly one has $$h_{2}\geq H/2\geq S(n)d.$$
It follows that $$h=\min (h_{1},h_{2})\geq S(n)d.$$
\end{proof}

\subsection{A volume's ratio}
Continuing with notations $p, d$ in Lemma \ref{l2.3}, we shall prove

\begin{lemma}\label{l2.4}
We have the following ratio of volume
$$\frac{V_{p}(2d)}{V_{p}(d)}\leq (1-(\frac{2S(n)}{3})^{n})(2^{n}-1)+1.$$
\end{lemma}

Before giving the proof of Lemma \ref{l2.4}, (following \cite{[Sh]}) we introduce some necessary notations. Let $\Sigma_{p}$ be a close subset of unit tangent sphere $S_{p}M\subset T_{p}M$. Let $B_{\Sigma_{p}}(r)$ be the set of points $x\in B_{p}(r)$ such that there exists a minimal geodesic $\gamma$ from $p$ to $x$ with $\gamma^{'}(0)\in \Sigma_{p}$. We write $V_{\Sigma_{p}}(r)$ for the volume of $B_{\Sigma_{p}}(r)$.

We denote by $\Sigma_{p}(r)$ the set of unit vectors $v\in S_{p}M$ such that $\gamma(t)=\exp_{p}(tv)$ is minimal on $[0,r)$.

\begin{proof}\emph{of Lemma \ref{l2.4}}.
We write $y=\gamma(d/2)$. Since $h\leq d/2$, we have $B_{p}(d)\supset B_{y}(h)\cup B_{\Sigma_{p}(d)}(d).$
By the definition of $h$, this gives $V_{p}(d)\geq V_{y}(h)+V_{\Sigma_{p}(d)}(d),$
i.e.
\begin{equation}\label{f2.2}
1\geq\frac{V_{y}(h)}{V_{p}(d)}+\frac{V_{\Sigma_{p}(d)}(d)}{V_{p}(d)}.
\end{equation}
\textbf{Claim 1:}
\begin{equation}\label{f2.3}
\frac{V_{y}(h)}{V_{p}(d)}\geq (\frac{2S(n)}{3})^{n}.
\end{equation}
By the Bishop-Gromov comparison theorem, $\mu^{n}V_{p}(r)\geq V_{p}(\mu r)$ for $\mu\geq1$. By Lemma \ref{l2.3} we have $h\geq S(n)d$.
So $$V_{y}(h)\geq V_{y}(S(n)d)\geq (\frac{2S(n)}{3})^{n}V_{y}(3d/2).$$
Since $B_{y}(3d/2)\supset B_{p}(d)$, we have $V_{y}(3d/2)\geq V_{p}(d)$. Thus we obtain (\ref{f2.3}).\\
\textbf{Claim 2:}
\begin{equation}\label{f2.4}
\frac{V_{\Sigma_{p}(d)}(d)}{V_{p}(d)}\geq \frac{1}{2^{n}-1}(\frac{V_{p}(2d)}{V_{p}(d)}-1).
\end{equation}
Following the observation of Shen (c.f. \cite{[Sh]} Lemma 2.4), we see that
$$B_{p}(2r)\setminus B_{p}(r)\subset B_{\Sigma_{p}(r)}(2r)\setminus B_{\Sigma_{p}(r)}(r).$$
Then we have
\begin{eqnarray*}
V_{p}(2r)-V_{p}(r) &\leq & V_{\Sigma_{p}(r)}(2r)-V_{\Sigma_{p}(r)}(r)\\
&\leq& (2^{n}-1)V_{\Sigma_{p}(r)}(r).
\end{eqnarray*}
The second in equality follows from the generalized volume comparison (Lemma 2.2 of \cite{[Sh]}).
Thus $$\frac{V_{\Sigma_{p}(r)}(r)}{V_{p}(r)}\geq \frac{1}{2^{n}-1}(\frac{V_{p}(2r)}{V_{p}(r)}-1).$$

Jointing formulas (\ref{f2.2}), (\ref{f2.3}) and (\ref{f2.4}), we establish the lemma.
\end{proof}

\section{A proof of theorem \ref{t1.1}}

We set $$C(n)=(1-(\frac{2S(n)}{3})^{n})(2^{n}-1)+1.$$

If $\frac{V_{p}(2r)}{V_{p}(r)}> C(n)$ for all $r>0$, then there is no nontrivial generator satisfying Lemma \ref{l2.4}. So $M$ is simple connected. Thus the first part of Theorem \ref{t1.1} is proved.

The proof of second part of Theorem \ref{t1.1} is divided into two steps.

Firstly, $\pi_{1}(M,p)$ is finitely generated. We argue by contradiction. Assume  $\pi_{1}(M,p)$ is infinitely generated, then by Lemma \ref{l2.2}, there is a sequence $\{d_{k}\}$, $d_{k}\rightarrow \infty$ as
$k\rightarrow \infty$ satisfying Lemma \ref{l2.4}, i.e.
$$\frac{V_{p}(2d_{k})}{V_{p}(d_{k})}\leq C(n),$$
for all $k\geq1$. This contradicts to condition (\ref{f1.2}).

Secondly, condition (\ref{f1.2}) implies that $V_{p}(r)\geq C\cdot r^{n-\epsilon}$ for some $\epsilon<1$ and sufficiently large $r$. So by Anderson's result \cite{[A]}, $\pi_{1}(M)$ has polynomial growth of order $\leq \epsilon<1$.

The form of $C(n)$ allows us to write $C(n)=2^{n-\epsilon}, \epsilon<1$. By condition (\ref{f1.2}), there exists $r_{0}>0$, for all $r\geq r_{0}$, one has
$$V_{p}(2r)\geq 2^{n-\epsilon} V_{p}(r).$$
So $$V_{p}(r_{0})\leq (2^{n-\epsilon})^{-1} V_{p}(2r_{0})\leq\cdot\cdot\cdot\leq (2^{n-\epsilon})^{-k} V_{p}(2^{k}r_{0}),$$
$k\in N$. For any $r\geq r_{0}$, we can assume $r\in[2^{k}r_{0},2^{k+1}r_{0}]$ for some $k\in N$. Then
\begin{eqnarray*}
V_{p}(r) \geq  V_{p}(2^{k}r_{0}) &\geq & \frac{V_{p}(r_{0})}{r_{0}^{n-\epsilon}}(2^{k}r_{0})^{n-\epsilon}\\
&\geq & \frac{V_{p}(r_{0})}{r_{0}^{n-\epsilon}}(2^{k}r_{0}\cdot\frac{r}{2^{k+1}r_{0}})^{n-\epsilon}\\
&= & \frac{V_{p}(r_{0})}{(2r_{0})^{n-\epsilon}}r^{n-\epsilon}.
\end{eqnarray*}
It follows that $V_{p}(r)\geq C\cdot r^{n-\epsilon}$ for all $r>r_{0}$.

 An algebraic fact:  If $\Gamma$ is an infinite group with generators $S=\{g_{1},\cdot\cdot\cdot,g_{k}\}$, then $\sharp U(r)\geq r$ for all $r\in N$, where $U(r)$ is the set of elements with word length $\leq r$ with respect to $S$. In particular, $\Gamma$ has polynomial growth of order at least 1. (This proof is provided by referee) To see this we argue by contradiction. Let $r$ be the smallest integer so that $\sharp U(r)< r$, then $r-1 \leq \sharp U(r-1) \leq\sharp U(r)<r$. This shows $U(r) = U(r-1)$. In other words, any word of length $r$ can be expressed as a word of length $\leq r-1$. It follows that $\Gamma=U(r-1)$, which is finite, a contradiction.

The second part of Theorem \ref{t1.1} follows from above immediately.

\begin{remark}
Our proof of finite generation of $\pi_{1}(M)$ is much different to Wu's arguments under condition (\ref{f1.3}).  Wu's proof was based on the estimate of ordered set of independent generators with minimal representative geodesic loops.
\end{remark}

\bibliographystyle{amsplain}

\end{document}